\documentclass{amsart}
\usepackage[utf8]{inputenc}
\usepackage[all,cmtip]{xy}
\usepackage{enumerate}
\usepackage{amsfonts,amssymb}
\usepackage{float}
\usepackage{graphicx,tikz}
\usepackage{tikz-cd} 
\usepackage{float}
\usepackage{amsmath, amsthm, amsfonts}
\usepackage{geometry}
\usepackage{hyperref}
\usepackage{enumitem}  
\usepackage[mathscr]{euscript}
\usepackage[capitalise]{cleveref}
\usepackage{enumitem}
\usepackage{amsrefs}
\usepackage[official]{eurosym}
\usepackage{mathtools}
\usepackage{graphicx}
\usepackage{nomencl}
\usepackage{amsfonts}
\usepackage{hyperref}
\usepackage{amssymb}
\usepackage{fancyhdr}
\usepackage{amscd}
\usepackage[english]{babel}
\usepackage{easybmat}
\usepackage{multirow,bigdelim}
\usepackage{comment}
\theoremstyle{definition}

\newtheorem*{theorem*}{Theorem}
\newtheorem{thm}{Theorem}

\newtheorem{cor}[thm]{Corollary}
\newtheorem{lem}[thm]{Lemma}
\newtheorem{theorem}{Theorem}

\newtheorem{corollary}[theorem]{Corollary}

\newtheorem{dfn}[thm]{\scshape{Definition}}

\newcommand\M{{\mathcal{M}}}
\newcommand\T{{\mathcal{T}}}
\newcommand\N{{\mathbb{N}}}
\newcommand\LL{{\mathcal L}}

\DeclareMathOperator{\Sym}{Sym}
\DeclareMathOperator{\Aut}{Aut}

\title{Tree languages and branched groups}
\author[L. Bartholdi]{Laurent Bartholdi}
\address{Laurent Bartholdi: Fakultät für Mathematik und Informatik, Universität des Saarlandes, 66123 Saarbrücken, Germany}
\email{laurentbartholdi@gmail.com}
\author[M. Noce]{Marialaura Noce}
\address{Marialaura Noce: Mathematisches Institut, Georg-August Universit\"{a}t zu G\"{o}ttingen, 37073 G\"{o}ttingen, Germany}
\email{mnoce@unisa.it}
\date{\today}
\keywords{Groups acting on trees, branched groups, tree languages, regular tree languages}
\subjclass[2020]{20E08, 68Q45, 20F10}

\begin{document}

\maketitle

\begin{abstract}
We study the portraits of isometries of rooted trees --- the labelling of the tree, at each vertex, by the permutation of its descendants --- in terms of languages. We characterize regularly branched self-similar groups in terms of $\omega$-regular languages. We deduce the algorithmic decidability of some problems, such as the comparison of regularly branched contracting groups, and their orbit structure on the boundary of the rooted tree.
\end{abstract}

\section{Introduction}
There is a rich interplay between group theory and theoretical computer science, and more precisely the theory of formal languages. We develop it in the context of \emph{self-similar groups} acting on rooted trees.

\subsection{Regular languages}

Let $A$ be a finite set called  \emph{alphabet}. A \emph{language} is a subset of $A^*$, where $A^*$ denotes the set of all finite words in $A$. Some languages can be described by a \emph{finite state automaton}: a device $\M$ with a finite amount of memory that reads letter-by-letter words over the alphabet $A$ as input and decides at the end whether to accept the word. 
An automaton $\M$ \emph{accepts} the word $w$ if there exists a run  $q_0, \dots, q_n$ on $w$ such that $q_n\in Q_f$. The language $\LL(\M)$ is the set of words accepted by the automaton $\M$. A language $K$ is \emph{regular} if there exists a finite state automaton $\M$ such that $K = \LL(\M)$. For more information we refer to~\cite{automatabook}.

Consider a group $G$ generated (as a monoid) by a set $A$. The set of words in $A^*$ that represent the identity in $G$ is a language $W(G)$, called the \emph{word problem} of $G$, and determines $G$. It is a still active area of research to relate the complexity of $W(G)$, as a language, with structural properties of $G$. The first result in this direction is due to Anisimov, who in 1970 proved that $W(G)$ is a regular language  if and only if the group $G$ is finite~\cite{MR301981}.

The group $G=\langle A\rangle$ may also be understood via a \emph{normal form} for its elements: a language $N(G)\subseteq A^*$ that is in bijection with $G$ under the natural evaluation map; for example, the free Abelian group $\mathbb Z^2$ can be represented by the regular language $(\{x\}^*\cup\{x^{-1}\}^*)(\{y\}^*\cup \{y^{-1}\}^*)$.

Automata may also be used to represent relations: adding an extra padding symbol $\diamond$ to $A$, one may represent a relation $R\subseteq A^*\times A^*$ by its language
\[
\{(u\diamond^{\max(0,|v|-|u|)},v\diamond^{\max(0,|u|-|v|)}:(u,v)\in R\},
\]
and the relation $R$ is called \emph{regular} if its language is regular as a subset of $((A\sqcup\{\diamond\})\times(A\sqcup\{\diamond\}))^*$. The rich class of \emph{automatic groups} are those groups $G=\langle A\rangle$ admitting a regular normal form $N\subseteq A^*$ and for which the multiplication relations $\{(u,v)\in N^2:u a=_G v\}$ are regular for all $a\in A$; see~\cite{cannonepstein}.

Kharlampovich, Khoussainov, and Miasnikov introduced Cayley-automatic groups in~\cite{cayleyautomatic} as groups admitting a regular normal form (not necessarily coming from words over a generating set) for their elements and the group operation. They prove that this class of groups strictly extends that of automatic groups, while retaining some of the most important algorithmic properties (such as fast solvability of the word problem). In this paper, the groups we are interested in shall not admit regular normal forms in the above sense; but rather as ``tree languages'', see below.

The theory of regular languages has been extended to infinite words by Büchi, see~\cite{MR125010}. In contrast to classical finite automata, a Büchi machine is a finite state automaton that takes infinite words as input and and accepts a run if the set of recurring states intersects the set of accepting states. An \emph{$\omega$-regular} language is a language recognized by a Büchi machine. The class of $\omega$-regular languages is closed under all boolean operations; this was used to prove decidability of Pressburger arithmetic~\cite{pressburgeraarithmetic}.

The theory was further extended to \emph{tree languages}. Finite words are seen as $A$-labels on the vertices of a finite path, and $\omega$-words are $A$-labelled infinite paths. In this vein, \emph{tree words} are $A$-labelled rooted trees (infinite in our case), and a \emph{tree automaton} is a finite machine accepting certain tree words. It reads trees starting from the root, one node at a time, and clones itself to process independently and in parallel all descendant subtrees. These tools are at the heart of Rabin's proof~\cite{rabin} that the monadic second-order logic of two successor functions (S2S) is decidable.

\subsection{Self-similar groups}
Finite state automata also appear in group theory as representations of the actual \emph{elements} of a group. For an alphabet $X$, the set of words $X^*$ naturally forms the vertex set of a regular rooted tree $\T$, rooted at the empty word and with an edge between $u$ and $u x$ for all $u\in X^*,x\in X$. Consider a group $G$ acting by tree isometries on $\T$, and $g\in G$. Then the action of $g$ on $\T$ is entirely characterized by its action $\sigma_g$ on $X$, the neighbours of the root, and for each $x\in X$ the induced bijection between the subtrees $x X^*$ and $\sigma_g(x)X^*$. Identifying these subtrees with $\T$, we obtain a permutation $\sigma_g$ and a collection $(g@x)_{x\in X}$ of tree isometries called the \emph{states} of $g$. If $g@x\in G$ for all $x\in X$, we call the group $G$ \emph{self-similar}. In this case, the action of $G$ on $\T$ is conveniently recorded by a homomorphism $\psi\colon G\to G\wr X$, defined by $g\mapsto((g@x),\sigma_g)$, where $G\wr X$ stands for the permutational wreath product $G^X\rtimes\Sym(X)$.

Furthermore, the operations $g\mapsto g@x$ may be composed; for $v=u x\in X^*$ we write $g@v=(g@u)@x$. If $g$ is such that $\{g@v:v\in X^*\}$ is finite, then $g$ may be represented by a finite automaton with alphabet $X\times X$: its stateset is $Q=\{g@v:v\in X^*\}$ and it has for all $h\in Q,x\in X$ a transition from $h$ to $h@x$ with label $(x,\sigma_h(x))$. 
As an example, consider $X=\{0,1\}$, the stateset $Q=\{a, b, c, d, e\}$, and the transitions
\begin{xalignat*}{2}
(a,(0,1))&\mapsto e, &(a,(1,0))&\mapsto e,\\
(b,(0,0))&\mapsto a, &(b,(1,1))&\mapsto c,\\
(c,(0,0))&\mapsto a, &(c,(1,1))&\mapsto d,\\
(d,(0,0))&\mapsto e, &(d,(1,1))&\mapsto b,\\
(e,(0,0))&\mapsto e, &(e,(1,1))&\mapsto e.
\end{xalignat*}
Every choice of initial state $q\in Q$ yields a relation $R_q\subset X^*\times X^*$ which is the graph of a permutation of $X^*$; the group generated by these permutations is known as the (first) Grigorchuk group~\cite{MR565099} and possesses a wealth of striking properties: it is a finitely generated torsion group which is infinite but all its proper quotients are finite; it is a group of intermediate word growth; it is amenable but not elementarily amenable, \dots. The Grigorchuk group and the Gupta-Sidki $p$-groups are instances of \emph{GGS-groups}, and provide a large family of counterexamples to the General Burnside Problem. 

Important subclasses of groups acting on trees have emerged: \emph{recurrent groups} are those for which the map `$g\mapsto g@v$' is surjective from the stabilizer of $v\in\T$ onto $G$. \emph{Level-transitive groups} are those that act transitively on $X^n$ for all $n\in\N$. Loosely speaking, a branched groups is a group whose  lattice of subnormal subgroups is similar to the corresponding structure in the full isometry group $\Aut\T$ of the tree $\T$. For a self-similar group $G$, there is a stronger notion: $G$ is \emph{regularly branched} if there is a finite-index subgroup $K$ such that $G$ contains, for all $v\in\T$, a copy of $K$ acting only on the subtree $v X^*$. The aforementioned examples of the Grigorchuk and Gupta-Sikdi torsion groups are recurrent, level-transitive and regularly branched.

Every $f \in \Aut\T$ can be described by providing, at every vertex $v$ of the tree, a permutation $\sigma_{f@v}\in\Sym(X)$ that describes how $f$ acts on the children of $v$. The \emph{portrait} of $f$ is this $\Sym(X)$-labelling $v\mapsto\sigma_{f@v}$ of $\T$, and there is a one-to-one correspondence between isometries of $\T$ and portraits.

The relations between a self-similar group $G$ and the collection of portraits of its elements have already been considered in the literature, in particular by Siegenthaler in~\cite{MR2429355} and his 2008 doctoral thesis; see also~\cites{MR2923476,MR2823790,MR3436774}. In that last article, Penland and \v{S}uni\'{c} prove that the closures of certain self-similar groups of rooted trees that satisfy an algebraic law do not have a regular language of portraits.

\subsection{Main results and their consequences}
This paper's main results relate tree languages with group-theoretical properties of branched groups.

\begin{theorem}[= Theorem~\ref{thm1}]\label{thmA}
Let $G$ be a level-transitive, recurrent, closed self-similar subgroup of $\Aut\T$. Then $G$ is regularly branched if and only if its set of portraits is a tree regular language.
\end{theorem}

We point out that in the ``if'' direction of the theorem above, we do not need that $G$ is level-transitive. 

\begin{theorem}[= Theorem~\ref{thm2}]\label{thmB}
Let $G$ be a contracting regular branched group. Then its set of portraits is a tree regular language.
\end{theorem}

Theorems~\ref{thmA} and~\ref{thmB} solve some algorithmic decidability problems in regularly branched groups:
\begin{corollary}
Let  $G, H \leq \Aut\T$ be two contracting regularly branched groups; then there exists an algorithm that determines whether $G \leq H$. As a consequence, it is decidable if $G = H$.
\end{corollary}

For a group $G$ acting on a tree $X^*$ and therefore on its boundary $X^\omega$, the \emph{orbit relation} is the equivalence relation
\[\mathcal O=\{(\xi,\eta)\in X^\omega\times X^\omega:\eta\in G\xi\}\subseteq (X\times X)^\omega.\]
\begin{corollary}
Let $G \leq \Aut\T$ be a contracting, regularly branched group; then its orbit relation is an $\omega$-regular language, and is computable.
\end{corollary}

The \emph{Hausdorff dimension} of profinite groups was first considered by Abercrombie in~\cite{MR1281541}. In our situation of a closed group $G$ acting on a rooted tree $X^*$, its Hausdorff dimension is computed as follows: let $G_n$ be the natural quotient of $G$ acting on $X^n$; then
\[\operatorname{Hdim}(G)=\liminf_{n\to\infty}\frac{\log\#G_n}{\log\#(\Aut X^*)_n}.\]
\begin{corollary}
Let $G \leq \Aut\T$ be a closed self-similar regularly branched group; then the Hausdorff dimension of $G$ is computable.
\end{corollary}

The paper is organized as follows: in Section~\ref{Section:branchgroups} we define groups acting on trees, with a particular emphasis on branched groups. Moreover, we consider branched groups in the more abstract setting of bisets. In Section~\ref{Section:treelanguages}, we give basic notions about tree languages, and in the last Section~\ref{Section:mainresults} we prove theorems~\ref{thmA} and~\ref{thmB} together with their corollaries.

\section{Branched groups}\label{Section:branchgroups}

In this Section we briefly summarize and recall the most important notions in the theory of groups of isometries of rooted trees and, more precisely, of branched groups. For more information on the topic see, for example, \cites{branch,delaharpe, baumslag}.

\subsection{Groups of isometries of rooted trees}\label{rootedtrees}
Let $\T$ be a \textit{$d$-regular rooted} tree: an oriented tree with a designated vertex called the root, and such that every vertex has $d$ immediate descendants and one immediate ancestor (except the root which has no ancestor).

Let $X$ be a finite alphabet with $d$ elements, $X^{n}$ the set of all words of length $n$ and $X^* = \bigcup_{n \in \mathbb{N}} X^{n}$; then the vertex set of $\T$ may be identified with $X^*$. We denote by $\epsilon$ the empty word, namely the root of the tree. The descendants of $u\in X^n$ are all words $u x\in X^{n+1}$, for $x\in X$. For every vertex $u\in X^*$, the subtree of $\T$ hanging from $u$ has vertex set $u X^{*}$ and is naturally identified with $\T$ via $uw\leftrightarrow w$. 
Let $\Aut\T$ denote the group of bijective maps of the set $X^{*}$ which preserve incidence, so are isometries of the graph $\T$. Let $\Sym(X)$ be the set of all permutations over the alphabet $X$. Considering separately the action of $\Aut\T$ on $X$ and on each subtree $x X^*$ for $x\in X$, we obtain an isomorphism
\[
\psi\colon \Aut\T \to  (\Aut\T)\wr X\coloneqq(\Aut\T)^X\rtimes\Sym(X)
\]
decomposing every isometry into a permutation of $X$ and a tuple of elements of $\Aut\T$. For $g\in\Aut\T$, we write $\psi(g)_1$, respectively $\psi(g)_2$ the elements of $\Aut\T^X$, respectively $\Sym(X)$; then $\psi(g)_1(x)$ is the map $v\mapsto x v\overset g\mapsto g(x v)=y w\mapsto w$, and $\psi(g)_2$ is the map $x\overset g\mapsto g(x)$.

\begin{dfn}[Portraits]\label{def:portrait}
Anticipating~\S\ref{Section:treelanguages}, a \emph{tree word} is a map $\T\to A$ for some alphabet $A$.
The \emph{portrait} of $g\in \Aut\T$ is the tree word $p_g\colon\mathcal T\to\Sym(X)$ defined inductively by $p_g(\epsilon)=\psi(g)_2$ and $p_g(x v)=p_{\psi(g)_1(x)}(v)$. The portrait map induces a homeomorphism between $\Aut\T$ and the set of all tree words $\mathcal T\to\Sym(X)$.
\end{dfn}

When considering subgroups $G \leq \Aut\T$, the following definitions clarify additional properties of the restriction to $G$ of the map $\psi\colon \Aut\T \to  \Aut\T \wr X$.

\begin{dfn}[Self-similar and recurrent groups]
Consider $G \leq \Aut\T$. It is \emph{self-similar} if $\psi$ restricts to a map $G\to G\wr X$, and \emph{recurrent} if furthermore for every $g\in G,v\in X^*$ there exists $h\in G$ with $h(v)=v$ and $h@v=g$.
\end{dfn}

For $g\in\Aut\T$ and $v\in X^*$, the \emph{state} of $g$ at $v$, written $g@v$, is defined inductively by $g@\epsilon=g$ and $g@(x v)=\psi(g)_1(x)@v$. Then $G$ is self-similar precisely when $g@v\in G$ for all $g\in G,v\in X^*$, and is recurrent if $(G_v)@v=G$ for all $v\in X^*$.

The portrait of $g\in\Aut\T$ may then be defined directly by $p_g(v)=\sigma_{g@v}$. Note that if $G$ is self-similar then its set of portraits is a ``differentially closed'' language, namely a language $\LL$ such that for every $w\in \LL$ and every $x\in X$ one has $\partial w/\partial x\in \LL$, where $(\partial w/\partial x)(v)=w(x v)$.

\begin{dfn}[Level-transitive groups]
A group $G$ acting on $\mathcal T$ is \emph{level-transitive} if it acts transitively on $X^n$ for all $n\in\mathbb N$; so the quotient of $\mathcal T$ by the action is a single ray.
\end{dfn}

\begin{dfn}[Contracting groups]\label{contracting}
A self-similar group $G \leq \Aut\T$  is \textit{contracting} if there exists a finite subset $F \subseteq G$ such that for every $g \in G$ the state $g@v$ belongs to $F$ for all $v\in X^*$ except finitely many. There is a minimal such set $F$ called the \emph{nucleus} of $G$, and denoted by $\mathcal{N}$.
\end{dfn}

\subsection{Branched groups and branch structure}
In this section we describe one of the most important classes of groups of isometries of rooted trees: branched groups. 

\begin{dfn}[Regularly branched groups]
A group $G$ acting on $\mathcal T$ is \emph{regularly branched} if there exists a finite index subgroup $K\le G$ with $K^X\le\psi(K)$.
\end{dfn}

\noindent This definition can be conveniently expressed as the existence of a \emph{branch structure}: a finite quotient $Q$ of $G$, a subgroup $Q_1\le Q\wr X$, and epimorphisms $\pi\colon G\twoheadrightarrow Q$ and $\phi\colon Q_1\twoheadrightarrow Q$. Indeed, with $K$ be as above assumed without loss of generality normal, set $Q\coloneqq G/K$ and $Q_1\coloneqq\psi(G)/K^X$, with $\pi\colon G\to Q$ the natural map and $\phi\colon \psi(G)/K^X\to \psi(G)/\psi(K)\cong G/K$.

\subsection{Branched bisets}\label{branchedbiset}

There exists a formalism of groups acting on trees that does not require fixing a model $X^*$ for the regular rooted tree: it is that of \emph{bisets}. A \emph{$G$-$G$-biset} is a set $B$ equipped with two commuting $G$-actions, written on the left and the right respectively. A \emph{covering biset} is a biset such that the right action is free, so one may write $B=X\cdot G$ when only considering the right action, by choosing a set of representatives $X\subseteq B$ of the right action. This set $X$ is called a \emph{basis} of $B$. The left action is then expressed as a map $G\times X\to X\times G$.

If a group $G$ acts self-similarly on a tree $\mathcal T=X^*$, its associated biset is $B=X\times G$ with left action given by $g\cdot(x,h)=(\psi(g)_2(x),\psi(g)_1(x)h)$. In other words, the map $\psi\colon G\to G\wr X$ induces naturally a map $G\to (X\times G)^X$ and therefore a map $G\times X\to X\times G$, which precisely defines the structure of a covering biset.

Conversely, given covering $G$-$G$-bisets $B$ and $C$, define their product $B\otimes_G C=(B\times C)/\{(b g,c)=(b,g c)\}$, and note that it is also a covering biset. More precisely, if $B\cong X\times G$ and $C\cong Y\times G$ qua right $G$-sets, then $B\otimes_G C\cong X\times Y\times G$ qua right $G$-set. Define the \emph{Fock tree} of $B$ as $\mathcal T=\bigsqcup_{n\ge0} B^{\otimes n}/G$, with natural left $G$-action. Note that if $B\cong X\times G$ then $\mathcal T\cong X^*$. We stress the point that $B$ determines canonically an action of $G$ on $\mathcal T$, but not on $X^*$ which requires a choice of identification $B\cong X\times G$. In fact, two different choices yield conjugate actions, but conjugate actions need not correspond to isomorphic bisets.

\begin{dfn}[Regularly branched bisets]
A $G$-$G$-biset $B$ is \emph{regularly branched} if there exists a finite index subgroup $K\le G$ and a basis $X$ of $B$ such that, for every $(k_x)\in K^X$, there exists $k\in K$ with $x k_x=k x$ for all $x\in X$.
\end{dfn}

\begin{lem}
The following hold:
\begin{enumerate}[label=(\roman*)]
    \item If $G$ acts on $X^*$ and is regularly branched, then its associated biset is regularly branched.
    \item If $B$ is regularly branched, then for every identification $B\cong X\times G$ the action on $X^*$ is regularly branched.
\end{enumerate}
\end{lem}
\begin{proof}
(i) Since $G$ is regularly branched, there exists a finite index subgroup $K\le G$ with $\psi(K^X)\le K$. Given now $(k_x)\in K^X$, set $k=\psi(k_x)$; we have $k x=x k_x$ as required. \\
(ii) Assume $B$ is regularly branched, for the subgroup $K$ which without loss of generality we suppose normal, and the basis $X$. Consider another basis $Y$. Then there exists a bijection $\pi\colon X\to Y$ and group elements $(g_x)\in G^X$ such that $x=\pi(x)g_x$ holds for all $x\in X$. Given now $(k_y)\in K^Y$, note that $(k_{\pi(x)}^{g_x})\in K^X$ because $K$ is normal. Let $k\in K$ be such that $x k_{\pi(x)}^{g_x}=k x$ for all $x\in X$, and compute, for $y=\pi(x)$,
\[y k_y = x g_x^{-1} k_y = x (k_y)^{g_x} g_x^{-1} = k x g_x^{-1} = k y;\]
so the ``there exists a basis'' in the definition of ``$B$ is regularly branched'' can be changed to a ``for every basis''.
\end{proof} 

In this section we have merely noted that the property of being regularly branched is independent of the choice of basis, so may be defined for bisets rather than tree actions. However, it would be more satisfying to express the definition more intrinsically, without any appeal to a basis. It could also be that the definition diverges from that of tree actions in case the group $G$ does not act faithfully. We leave these questions open for further investigation; the general point being that a biset should be branched if and only if the left action has finitely many orbits on bases. 

\section{Automata on infinite trees}\label{Section:treelanguages}

In this section we define automata on infinite trees, and defer to~\cite{tata} for a thorough introduction on automata on trees and tree languages.

Recall that an \emph{alphabet} is a finite set, and a \emph{regular tree} is a set of the form $\mathcal T=X^*$ for an alphabet $X$. Its vertices are thus finite sequences ($X^*$), and its boundary consists of infinite sequences ($X^\omega$) called \emph{rays}. For example, if $X=\{t\}$, the corresponding regular tree is geometrically an infinite half-line, with vertex set $\{t^n:n\in\N\}$, and it has a single ray $t^\omega$.

A \emph{tree word} is a map $w\colon\mathcal T\to A$, for a regular tree $\mathcal T$ and an alphabet $A$. For example, if $\mathcal T=\{t\}^*$, a tree word is a usual $\omega$-word. As we saw in Definition~\ref{def:portrait}, the portrait of a tree isometry is a tree word, with $A=\Sym(X)$.

A \emph{tree language} is a set of tree words, all with the same tree and alphabet; so again if a group acts on a tree $\mathcal T=X^*$ then the set of portraits of its elements is a tree language with alphabet $\Sym(X)$.

 \begin{dfn}[M\"uller automata]
 A \emph{M\"{u}ller tree automaton} on the tree $\T=X^*$ is a quintuple $\mathcal M = (Q, A, \delta, Q_i, Q_f)$, where
 \begin{itemize}
     \item $Q$ is a finite (non-empty) set of \emph{states};
     \item  $A$ is an alphabet;
     \item  $\delta \subseteq Q \times A \times Q^X$ is the \emph{transition rule};
     \item $Q_i \subseteq Q$ is the set of \emph{initial states};
     \item  $Q_f \subseteq 2^Q$ is the set of \emph{residual statesets}.
 \end{itemize}
 \end{dfn}
 
A \emph{run} of $\mathcal M$ is a tree word $w$ with alphabet $\delta$, such that $w(\epsilon)_1\in Q_i$ and for every $v\in X^*$ one has $w(v)_3(x)=w(v x)_1$ and for every ray $\xi\in X^\omega$ one has $\{w(v)_1:v\text{ prefix of }\xi\}\in Q_f$. In words, the run must start at an initial state, the states match along edges according to some rules in $\delta$, and the set of states seen along any ray is residual. The \emph{value} of a run $w$ is the tree word $\overline w$ with alphabet $A$ defined by $\overline w(v)=w(v)_2$. A tree word $\mathcal T\to A$ is \emph{accepted} by $\mathcal M$ if it is the value of a run of $\mathcal M$ (see ~\cite{mullercondition}).

\begin{dfn}[Regular tree languages]
A tree language is \emph{regular} if it is the set of accepted tree words of some Müller tree automaton.
\end{dfn}

We note the fundamental (and non-trivial) fact that the class of languages definable by non-deterministic M\"uller tree automata is closed under all boolean operations. 

\section{Main results}\label{Section:mainresults}
\noindent Here we prove Theorem~\ref{thmA} and~\ref{thmB} together with some applications.

\begin{thm}\label{thm1}
Let $G$ be a level-transitive, recurrent, closed self-similar subgroup of $\Aut(X^*)$. Then $G$ is regularly branched if and only if its set of portraits is a regular language.
\end{thm}
\begin{proof}
In the first direction, suppose $G$ is regularly branched, and let $(\pi\colon G\to Q, \phi\colon Q\wr X\ge Q_1\to Q)$ be a branch structure. Consider the following Müller automaton $\mathcal M$:
 \begin{itemize}
     \item $Q$ (defined above) is the set of states;
     \item  $A=\Sym(X)$;
     \item  $\delta \subseteq Q \times A \times Q^X$ with $(q,a,(q_x))\in\delta$ whenever $((q_x),a)$ defines an element $q_1\in Q_1$ and $\phi(q_1)=q$;
     \item $Q_i=Q$;
     \item  $Q_f \subseteq 2^Q$.
 \end{itemize}
 
Given $g\in G$, construct a run $w$ of $\mathcal M$ by $w(v)=(\pi(g@v),\sigma(g@v),(\pi(g@v x))_{x\in X})$, and note that its value is the portrait of $g$. Conversely, given a run $w$ of $\mathcal M$, note that its value defines a portrait and therefore an isometry $t$ of $X^*$. Furthermore, the run defines an element $w(\epsilon)_1\in Q$ so there exists $g\in G$ such that $g,t$ have same image in $Q$. Now the construction of $\mathcal M$ implies that for every $v\in X^*$ there exists $q_{1,v}=(w(v)_3,w(v)_2)\in Q_1$ with $\phi(q_{1,v})=w(v)_1$; thus there exists a lift $g_1\in G$ of $q_{1,\epsilon}$ such that $g_1,t$ have same image in $Q_1$. Continuing, there exists $q_{2,v}\in Q_2\le Q_1\wr X$ obtained by assembling $q_{1,v},(q_{1,v x})$, and a lift $g_2\in G$ of $q_{2,\epsilon}$ such that $g_2,t$ have same image in $Q_2$. In this way, we obtain a sequence of elements $g_n\in G$ such that $g_n,t$ have same image in $Q_n$. Since the $Q_n$ converge to $G$ and $G$ is closed, we get $t=\lim g_n\in G$.

Note that we did not use the hypotheses that $G$ be level-transitive nor recurrent.

In the other direction: assume that the language of portraits of the self-similar group $G$ is a regular language, and let $\mathcal M$ be a Müller machine recognizing it, with stateset $Q$. Without loss of generality, assume $\mathcal M$ is \emph{trim}: every state can be reached from an initial state, and leads to a residual set. For every $g\in G$, let $Q(g)$ denote the set of initial states of runs accepting the portrait of $g$, noting that $Q(g)$ is non-empty for all $g\in G$. Let $K_0$ denote the subgroup of $G$ generated by $\{g h^{-1}:Q(g)\cap Q(h)\neq\emptyset\}$.

We first note that $K_0$ is a finite index subgroup of $G$; more precisely, the index of $K_0$ is at most $\#Q$. Indeed, let $g_0,g_1,\dots,g_{\#Q}$ be arbitrary elements of $G$. Then the sets $Q(g_0),\dots,Q(g_{\#Q})$ cannot all be disjoint, since they are non-empty and more numerous than $Q$. Thus there exists $i$ and $j$ such that $Q(g_i)$ and $Q(g_j)$ intersect, so $g_i g_j^{-1}\in K_0$, and so $g_0,\dots,g_{\#Q}$ cannot be left coset representatives of $K_0$.

Consider next $k = g h^{-1}\in K_0$, and choose $q\in Q(g)\cap Q(h)$. Since $\mathcal M$ is trim, there exists an element $f\in G$ and a run $w$ accepting $f$ which reaches state $q$ at some vertex $v\in X^*$. Construct two runs $w_g,w_h$ by replacing in $w$ the subtree at $v$ respectively by the runs accepting $g,h$ in initial state $q$. These runs define respectively elements $f_g,f_h\in G$, both starting at the same initial state $r\in Q_i$. Therefore $Q(f_g)\cap Q(f_h)\ne\emptyset$, and thus $f_g f_h^{-1}\in K_0$. Furthermore, $f_g,f_h$ agree everywhere except on the subtree at $v$, so $f_g f_h^{-1}$ acts as $k$ on that subtree and trivially elsewhere. Thus, for every generator $k$ of $K_0$, we have constructed a vertex $v$ and an element $\ell\in K_0$ with $\ell@v=k$ and fixing the complement of $v X^*$.

We now let $K$ denote the normal closure of $K_0$ in $G$, and continue our reasoning with $\ell@v=k$. Write $|v|=n$, and note $n\le\#Q$ since we may have chosen a simple path from an initial state to $q$ in the construction of the run $w$. Since $G$ is recurrent and acts transitively on $X^n$, we obtain for arbitrary $u\in X^n,g\in G$ an element $\ell_{u,k}\in K$ acting as $k^g$ on the subtree $u X^*$ and fixing its complement.

Note that the choice of parameter $n$ depended only on the state $q$, so there exists a common multiple of all such choices, call it $m$. We deduce that $K^{X^m}$ is a subgroup of $K$, and $G$ is regularly branched with branching subgroup $K$, albeit on the tree $(X^m)^*$ obtained by coalescing $m$ levels into one.
\end{proof}

\noindent The following lemma will be useful for the proof of Theorem~\ref{thmB}.
\begin{lem}\label{preliminarylemma}
Let $g_x\in G$ and $a\in\Sym(X)$ such that there exists $g\in G$ with $\pi(\psi_1(g)(x))=\pi(g_x)$.
Then there exists $g'\in G$ with $\psi_1(g')(x)=g_x$.
\end{lem}
\begin{proof}
By hypothesis we have $k_x\coloneqq\psi_1(g)(x)^{-1}g_x\in K$ for all $x\in X$; and thus there exists $k\in K$ with $k@x=k_x$. Consider $g'=k g$, and compute $\psi_1(g')(x)=k_x\psi_1(g)(x)=g_x$.
\end{proof}

\begin{thm}\label{thm2}
If $G$ is a contracting regularly branched group, then its set of portraits is a regular language.
\end{thm}
\begin{proof}
We modify the construction in the proof of Theorem~\ref{thm1} as follows. Let $\mathcal N$ be the nucleus of $G$. As stateset, we use $Q'=Q\sqcup\mathcal N$. We set $Q'_i=Q'$ and $Q'_f=2^{\mathcal N}$. We extend the transitions $\delta$ to $\delta'$ by allowing transitions from $Q'$ to $\mathcal N$. More precisely, we add to $\delta$ for all $g\in G$ the transitions $(\pi(g),\sigma_g,(g@x)_{x\in X})$ whenever $g@x\in\mathcal N$ for all $x\in X$, and for all $g\in\mathcal N$ the transitions $(g,\sigma_g,(g@x)_{x\in X})$. Thus an accepting run must reach state $\mathcal N$ on every ray, and so by König's lemma may have only finitely many states in $Q'\setminus\mathcal N$.

On the one hand, it is clear that every $g\in G$ is accepted, since almost all states of $g$ belong to $\mathcal N$. Conversely, consider an accepting run. It yields by restriction a finite subtree of $X^*$, with elements of $\mathcal N$ at its leaves and elements of $Q\times A$ at its internal vertices. We construct by induction an element of $G$ from it. If the subtree has a single vertex, then it is labelled by an element of $\mathcal N$ and therefore directly produces an element of $G$. Otherwise, consider an internal vertex of maximal height, so all its descendants are labelled by elements of $G$. We may replace the internal vertex's label by an element of $G$, and proceed by induction, owing to Lemma~\ref{preliminarylemma}.
\end{proof}

Finally we list some consequences of Theorems~\ref{thmA} and~\ref{thmB}. These all pertain to decidability questions.

When given a contracting regularly branched group, the data defining it are assumed to be in the form of a finite generating set $S$, the restriction of $\psi$ to $S$, and an oracle promising that the action is faithful, contracting and regularly branched.
\begin{cor}
Let $G, H \leq \Aut\T$ be two contracting regularly branched groups. Then there exists an algorithm that determines whether $G \leq H$. As a consequence, it is decidable if $G=H$.
\end{cor}
\begin{proof}
Suppose that there exists an oracle that recognises whether a group is contracting and branched. Then one can algorithmically determine its nucleus and its branching subgroup. Now given two contracting regular branched groups  $G$ and $H$, by Theorem~\ref{thm2} their languages of portraits are regular, and the inclusion in regular languages is decidable. This completes the proof. 
\end{proof}

On the other hand, a closed regularly branched group is determined by the finite data of its branched structure.
\begin{cor}
Let $G \leq \Aut\T$ be a closed self-similar regularly branched group. Then the Hausdorff dimension of $G$ is computable.
\end{cor}
\begin{proof}
By Theorem~\ref{thm1}, there exists a tree automaton that recognizes elements of the portrait of a closed group $G$. The Hausdorff dimension of $G$ is the entropy of its language of portraits, so is computable from the automaton.

More directly, a linear system of equations can easily be set up, one variable per state of the automaton, whose solution yields the Hausdorff dimension of $G$, see~\cite{Siegenthaler2010}.
\end{proof}

We finally to turn to the orbit structure of contracting, regularly branched groups. Say the group $G$ acts on the regular tree $\T=X^*$; so its acts on its boundary $\partial\T=X^\omega$. The \emph{orbit relation} is the subset
\[\mathcal O=\{(\xi,\eta)\in X^\omega\times X^\omega:\eta\in G\xi\}\subseteq (X\times X)^\omega.\]
\begin{cor}
Let $G \leq \Aut\T$ be a contracting, regularly branched group. Then the equivalence relation $\mathcal O$ is an $\omega$-regular language, and is computable.
\end{cor}
\begin{proof}
Let $\mathcal M$ be a Müller tree automaton recognizing the portraits of $G$. Define a (classical) Müller automaton $\mathcal M'$ with same stateset $Q$, same initial and residual states $Q_f,Q_i$, alphabet $A'=X\times X$, and transitions $\delta'\subset Q\times A'\times Q$, by
\[\delta'=\{(q,(y,z),q_y):\exists (q,a,(q_x)_{x\in X})\in\delta\text{ with }a(y)=z\}.\]
Then every run of $\mathcal M$, recognizing an element $g\in G$ leads to runs, for all $\xi\in X^\omega$, of $\mathcal M'$ recognizing $(\xi,g\xi)$; thus the language of $\mathcal M'$ is $\mathcal O$.

In particular, for every preperiodic ray $\xi$, the orbit of $\xi$ is computable as an $\omega$-regular language $\subseteq X^\omega$, by intersecting $\mathcal O$ with the $\omega$-regular language $\{\xi\}\times X^\omega$ and projecting to the second coordinate.
\end{proof}

\bibliographystyle{plain}
\bibliography{bib}

\begin{bibdiv}
\begin{biblist}

\bib{MR1281541}{article}{
      author={Abercrombie, A.~G.},
       title={Subgroups and subrings of profinite rings},
        date={1994},
        ISSN={0305-0041},
     journal={Math. Proc. Cambridge Philos. Soc.},
      volume={116},
      number={2},
       pages={209\ndash 222},
         url={https://doi.org/10.1017/S0305004100072522},
}

\bib{MR301981}{article}{
      author={An\={\i}s\={\i}mov, A.~V.},
       title={The group languages},
        date={1971},
        ISSN={0023-1274},
     journal={Kibernetika (Kiev)},
      number={4},
       pages={18\ndash 24},
}

\bib{branch}{incollection}{
      author={{Bartholdi}, L.},
      author={{Grigorchuk}, R.~I.},
      author={{Sunik}, Z.},
       title={Branch groups},
        date={2003},
   booktitle={Handbook of {A}lgebra, {V}olume 3, {N}orth-{H}olland},
       pages={989\ndash 1112},
}

\bib{baumslag}{inproceedings}{
      author={Baumslag, G.},
       title={Topics in combinatorial group theory},
        date={1993},
   booktitle={Lectures in mathematics, {ETH} {Z}{\"u}rich},
}

\bib{pressburgeraarithmetic}{inproceedings}{
      author={Boudet, A.},
      author={Comon, H},
       title={Diophantine equations, presburger arithmetic and finite
  automata},
        date={1996},
   booktitle={Proceedings of the 21st international colloquium on trees in
  algebra and programming},
      series={CAAP '96},
   publisher={Springer-Verlag},
     address={Berlin, Heidelberg},
       pages={30–43},
}

\bib{MR125010}{article}{
      author={B\"{u}chi, J.~Richard},
       title={Weak second-order arithmetic and finite automata},
        date={1960},
        ISSN={0044-3050},
     journal={Z. Math. Logik Grundlagen Math.},
      volume={6},
       pages={66\ndash 92},
         url={https://doi.org/10.1002/malq.19600060105},
      review={\MR{125010}},
}

\bib{tata}{book}{
      author={Comon, H.},
      author={Dauchet, M.},
      author={Gilleron, R.},
      author={Jacquemard, F.},
      author={Lugiez, D.},
      author={L{\"o}ding, C.},
      author={Tison, S.},
      author={Tommasi, M.},
       title={{Tree Automata Techniques and Applications}},
        date={2008},
         url={https://hal.inria.fr/hal-03367725},
}

\bib{delaharpe}{book}{
      author={de~la Harpe, P.},
       title={Topics in geometric group theory},
      series={Chicago Lectures in Mathematics},
   publisher={University of Chicago Press},
        date={2000},
        ISBN={9780226317199},
         url={https://books.google.es/books?id=cRT01C5ADroC},
}

\bib{cannonepstein}{book}{
      author={Epstein, D. B.~A.},
      author={Paterson, M.~S.},
      author={Cannon, J.~W.},
      author={Holt, D.~F.},
      author={Levy, S.~V.},
      author={Thurston, W.~P.},
       title={Word processing in groups},
   publisher={A. K. Peters, Ltd.},
     address={USA},
        date={1992},
        ISBN={0867202440},
}

\bib{MR565099}{article}{
      author={Grigor\v{c}uk, R.~I.},
       title={On {B}urnside's problem on periodic groups},
        date={1980},
        ISSN={0374-1990},
     journal={Funktsional. Anal. i Prilozhen.},
      volume={14},
      number={1},
       pages={53\ndash 54},
}

\bib{automatabook}{book}{
      author={Hopcroft, J.~E.},
      author={Motwani, R.},
      author={Ullman, J.~D.},
       title={Introduction to automata theory, languages, and computation},
   publisher={Addison-Wesley Longman Publishing Co., Inc.},
     address={USA},
        date={2006},
        ISBN={0321455363},
}

\bib{cayleyautomatic}{article}{
      author={Kharlampovich, O.},
      author={Khoussainov, B.},
      author={Miasnikov, A.},
       title={From automatic structures to automatic groups},
        date={2011},
     journal={Groups, Geometry, and Dynamics},
      volume={8},
}

\bib{mullercondition}{inproceedings}{
      author={Muller, D.~E.},
       title={Infinite sequences and finite machines},
        date={1963},
   booktitle={Proceedings of the 1963 {P}roceedings of the {F}ourth {A}nnual
  {S}ymposium on {S}witching {C}ircuit {T}heory and {L}ogical {D}esign},
   publisher={IEEE Computer Society},
     address={USA},
       pages={3–16},
         url={https://doi.org/10.1109/SWCT.1963.8},
}

\bib{MR3436774}{article}{
      author={Penland, A.},
      author={\v{S}uni\'{c}, Z.},
       title={Finitely constrained groups of maximal {H}ausdorff dimension},
        date={2016},
        ISSN={1446-7887},
     journal={J. Aust. Math. Soc.},
      volume={100},
      number={1},
       pages={108\ndash 123},
         url={https://doi.org/10.1017/S144678871500018X},
}

\bib{rabin}{article}{
      author={Rabin, M.~O.},
       title={Decidability of second-order theories and automata on infinite
  trees},
        date={1968},
     journal={Bulletin of the American Mathematical Society},
      volume={74},
       pages={1025\ndash 1029},
}

\bib{MR2429355}{article}{
      author={Siegenthaler, O.},
       title={Hausdorff dimension of some groups acting on the binary tree},
        date={2008},
        ISSN={1433-5883},
     journal={J. Group Theory},
      volume={11},
      number={4},
       pages={555\ndash 567},
         url={https://doi.org/10.1515/JGT.2008.034},
}

\bib{Siegenthaler2010}{inproceedings}{
      author={Siegenthaler, O.},
       title={Discrete and profinite groups acting on regular rooted trees},
        date={2010},
}

\bib{MR2923476}{article}{
      author={Siegenthaler, O.},
      author={Zugadi-Reizabal, A.},
       title={The equations satisfied by {GGS}-groups and the abelian group
  structure of the {G}upta-{S}idki group},
        date={2012},
        ISSN={0195-6698},
     journal={European J. Combin.},
      volume={33},
      number={7},
       pages={1672\ndash 1690},
         url={https://doi.org/10.1016/j.ejc.2012.03.025},
}

\bib{MR2823790}{article}{
      author={\v{S}uni\'{c}, Z.},
       title={Pattern closure of groups of tree automorphisms},
        date={2011},
        ISSN={1664-3607},
     journal={Bull. Math. Sci.},
      volume={1},
      number={1},
       pages={115\ndash 127},
         url={https://doi.org/10.1007/s13373-011-0007-2},
}

\end{biblist}
\end{bibdiv}

\end{document}